\documentclass[11pt]{article}
\usepackage{dsfont}

\usepackage{psfrag}

\usepackage{calrsfs}
\usepackage{mathrsfs}
\usepackage{pdfsync}
\usepackage{amsmath, amsthm, amssymb, amsfonts,comment}
\usepackage[shortlabels]{enumitem}

\usepackage{url}
\usepackage{float}

\usepackage[usenames]{color}
\usepackage{tikz}

\usepackage[font={small}]{caption}
\usepackage[margin=1cm]{caption}

\usetikzlibrary{arrows,decorations.pathmorphing,backgrounds,positioning,fit,automata}

\usepackage{tikz,fullpage}
\usetikzlibrary{arrows}
\usetikzlibrary{petri}
\usetikzlibrary{topaths}

\usepackage{indentfirst,calc,euscript}
\usepackage{setspace}
\usepackage[reals]{layout}
\usepackage{xr}
\usepackage{amscd}

\usepackage{sgame}
\usepackage{subfigure}

\usepackage[bottom]{footmisc}

\usepackage[colorlinks=true,breaklinks=true,bookmarks=true,urlcolor=blue,
     citecolor=blue,linkcolor=blue,bookmarksopen=false,draft=false]{hyperref}

\newcommand{\1}{\mathbf 1}

\newcommand{\ignore}[1]{}

\renewcommand{\P}{\mathbb P}
\newcommand{\LL}{\mathbb L}
\newcommand*{\vcross}{{\setlength{\fboxsep}{0pt}\fbox{$\vert$}}}

\newtheorem{theorem}{Theorem}

\newtheorem{claim}[theorem]{Claim}

\newtheorem{corollary}[theorem]{Corollary}

\newtheorem{lemma}[theorem]{Lemma}

\newtheorem{proposition}[theorem]{Proposition}

\theoremstyle{definition}

\newtheorem{definition}[theorem]{Definition}

\newtheorem{remark}[theorem]{Remark}

\numberwithin{equation}{section}

\numberwithin{theorem}{section}

\newcommand{\m}{\mathbb}

\renewcommand{\thefootnote}{\fnsymbol{footnote}}

\author{
    Avelio Sepúlveda\footnote{Universidad de Chile, Centro de Modelamiento Matemático (AFB170001), UMI-CNRS 2807, Beauchef 851, Santiago, Chile.}
    \and 
    Bruno Ziliotto\footnote{CEREMADE, CNRS, PSL Research Institute, Paris Dauphine University, France.}
    }
\date{}
\title{The game behind oriented percolation}
\begin{document}
\maketitle
\abstract{
We characterize the critical parameter of oriented percolation on $\m{Z}^2$ through the value of a zero-sum game. Specifically, we define a zero-sum game on a percolation configuration of $\m{Z}^2$, where two players move a token along the non-oriented edges of $\m{Z}^2$, collecting a cost of 1 for each edge that is open, and 0 otherwise. The total cost is given by the limit superior of the average cost. We demonstrate that the value of this game is deterministic and equals 1 if and only if the percolation parameter exceeds $p_c$, the critical exponent of oriented percolation. Additionally, we establish that the value of the game is continuous at $p_c$. Finally, we show that for $p$ close to 0, the value of the game is equal to 0.
}

\renewcommand{\thefootnote}{\arabic{footnote}}
\section{Introduction}
Oriented percolation, introduced by Broadbent and Hammersley \cite{BH57}, is a variation of classic percolation where the edges of the lattice are oriented in a specific direction. This model was initially proposed to describe fluid propagation in a medium, such as electrons moving through an atomic lattice or water permeating a porous solid. Mathematically, oriented percolation is one of the simplest models exhibiting a phase transition and is closely related to the geometric representation of the contact process \cite{H74, H78}.

In this paper, we focus on oriented percolation on the graph $\mathbb{Z}^2$, rotated by $\pi/4$, with edges oriented to always ascend (see Figure \ref{f.realization}). For a fixed $p \in [0,1]$, each oriented edge is independently open with probability $p$. The phase transition of this model is characterized by a critical probability $p_c \in (0,1)$, such that for any $p \leq p_c$, almost surely, there is no infinite directed path of open edges, whereas for $p > p_c$, such a path almost surely exists. Despite extensive research, the exact value of $p_c$ remains unknown. Lower and upper bounds have been established in \cite{BRL06} and \cite{BBS94}, respectively, and heuristic estimates have been computed \cite{WZLGD13}.

\begin{figure}[h!]
    \centering
    \includegraphics[width=0.2\linewidth]{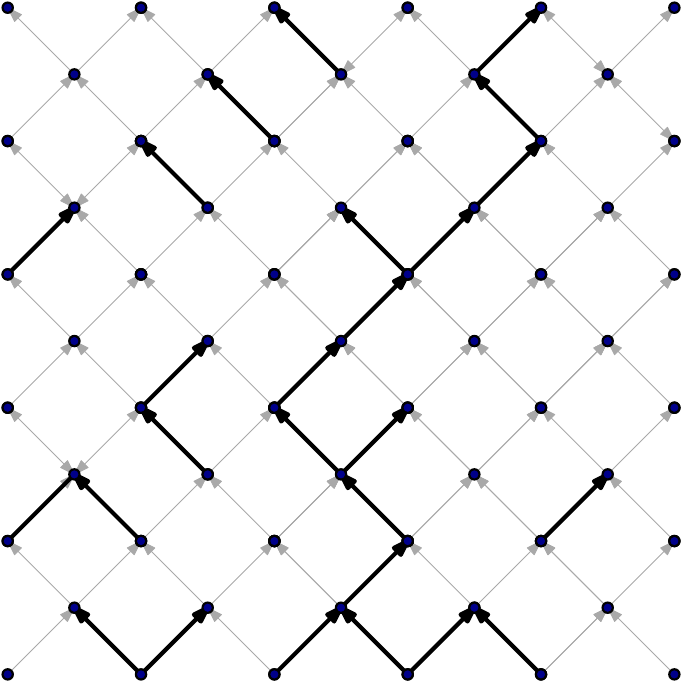}
    \caption{Diagram illustrating oriented percolation: all edges are oriented upwards, black edges are open and gray edges are closed. The image shows an upward crossing of the square.}
    \label{f.realization}
\end{figure}
The objective of this work is to characterize $p_c$ through a multi-stage zero-sum game played on $\mathbb{Z}^2$, where both players have full knowledge of the percolation configuration from the outset. Specifically, we construct a local, non-oriented game between two players, allowing movement in any of the four directions during each turn. We then show that this game experiences a phase transition at the same critical parameter $p_c$ as oriented percolation.

\newpage 

\subsection{Results}
Let us first define the game under study. We begin with a non-oriented percolation configuration on the graph $\LL$, which is the rotation\footnote{For a precise definition of the graph $\LL$, see Section \ref{s.preliminaries}.} by $\pi/4$ of $\m {Z}^2$. A token is placed at some vertex $z$ of $\LL$. At each stage, Player 1 selects either action $T$ (Top) or action $B$ (Bottom) and announces the choice to Player 2. Player 2 then chooses either action $L$ (Left) or action $R$ (Right). The token moves along the corresponding edge\footnote{That is to say, to the Top Right vertex if $(T,R)$ is played, to the Top Left vertex if $(T,L)$ is played, to the Bottom Left vertex if $(B,L)$ is played, and to the Bottom Right vertex if $(B,R)$ is played.}. If the chosen edge is open, Player 1 incurs a cost of 1 paid to Player 2; otherwise, the cost is 0. The game then proceeds to the next stage under the same rules.
Both players are fully aware of the configuration of open and closed edges before the game begins. The total cost is defined as
$\limsup_{n \rightarrow +\infty} \frac{1}{n} \sum_{m=1}^n c_m$, where $c_m$ is the cost at stage $m$. This game has a $\textit{value}$, denoted by $v_p(z) \in [0,1]$. The formal definition of the value is provided in Section \ref{sec:model}. This value can be interpreted as the \textit{solution} of the game in terms of cost, meaning that if both players play optimally, Player 1 should expect to incur a total cost of $v_p(z)$. 

We analyze the value function of this game,
providing a theoretical foundation for what is observed empirically in Figure \ref{f.graph}, and establishing a connection to the percolation properties of directed percolation on $\mathbb Z^2$.
\begin{figure}[!b]  \centering
    \includegraphics[width=0.5\linewidth]{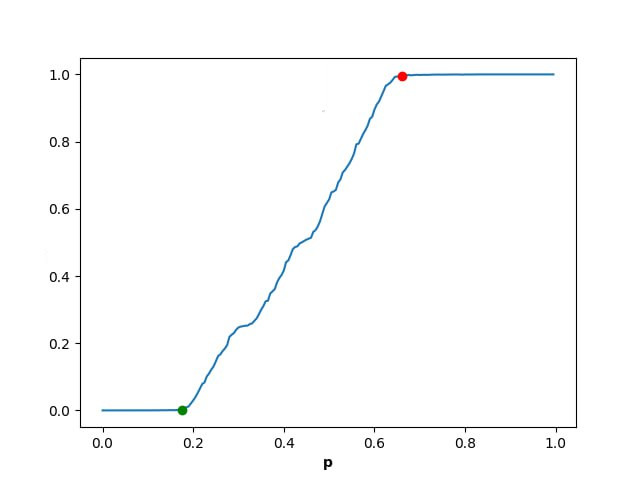}
    \caption{The graph presents simulations of the game value for each $p\in [0,1]$. The green point, $\sim 0.175$  represents the threshold below which the value is always $0$ (see Theorem \ref{t.v=0} for the non-triviality of this point). The red point, $\sim 0.66$, represents the threshold above which the value is always $1$; this point coincides with the critical parameter of oriented percolation in $\mathbb Z^2$ (see Theorem \ref{theo:value_threshold}). Note that both phase transitions appear to be continuous; we show that fact for the red point (see Theorem \ref{theo:continuity}). These simulations were performed by Melissa Garcia González, by playing the game up to 500 steps and 30 times for each $p$, where $p$ takes 200 equally spaced values in $[0,1]$. }
    \label{f.graph}
\end{figure}
The value $v_p(z)$ is, \textit{a priori}, a random variable, as its definition relies on the realizations of the Bernoulli variables that determine which edges are open. Our first result, Theorem \ref{theo:det}, establishes that $v_p(z)$ is actually independent of $z$ and almost surely deterministic. Therefore, the value function reduces to a number in $[0,1]$, which we now denote by $v_p$. Theorem \ref{theo:value_threshold}, our core result, asserts that $v_p=1$ if and only if $p \geq p_c$. This indicates that $p_c$ characterizes a phase transition in the game, where the value becomes 1. Theorem \ref{theo:continuity} demonstrates that the mapping $p \rightarrow v_p$ is continuous at $p_c$. Finally, in Theorem \ref{t.v=0}, we show that there exists $p_0\in (0,1)$ such that for all $p<p_0$, $v_p=0$.

A notable aspect of our result is that the game we define is not ``oriented'': players can move in any of the four directions. Despite this, the phase transition occurs at the same parameter as that of critical oriented percolation, rather than classical percolation. Another significant point is that at $p=p_c$, the value of the game is 1, while for critical oriented percolation, the probability of an infinite open path is 0.

Our results are motivated by both game-theoretic and probabilistic literature. From a game theory perspective, the game we define fits within the class of \textit{percolation games} introduced in \cite{GZ23} and extended in \cite{ALMSZ24}. While \cite{GZ23,ALMSZ24} focus on percolation games with an ``orientation'', our game offers a simple yet rich example of a non-oriented percolation game. It also connects to the broader literature on zero-sum stochastic games with long durations (see, e.g., \cite{sorin02b,LS15,SZ16,solan22,MSZ} for general references). Moreover, in our model, the players know the costs beforehand, which aligns our model with the concept of \textit{random games}, where a game is selected at random, revealed to the players, and then played \cite{alon2021dominancesolvabilityrandomgames, ACSZ21, FPS23, HJMP21}.

From the probabilistic literature perspective, our game generalizes the last-passage percolation (LPP) to a two-player setting, as LPP can be seen as a special case where Player 1 always chooses the top action. Given recent advances in understanding LPP and its connections to the KPZ universality class (see, e.g., \cite{GRS,DV,DOV}), our game presents a natural model to extend some of these findings. Furthermore, our main result resonates with \cite{PSSW}, which relates a game to (non)-oriented percolation. Another relevant work is \cite{HMM19}, where the authors explore a game in which two players move a token on the vertices of $\m{Z}^2$ towards a ``target'' while avoiding ``traps'' placed according to i.i.d. Bernoullis, using this framework to study the determinacy of a class of Probabilistic Finite Automata (see \cite{BKPR23} for an extension).

Finally, our work opens up several avenues for further exploration in both game theory and probability. As a matter of fact, one might consider alternative long-term cost structures or environments and investigate whether our game-theoretic characterization of $p_c$ can help establish new bounds or estimates for $p_c$. Additionally, it would be interesting to explore whether the near-critical exponents associated with our game for $p < p_c$ have any connection to the critical exponents of oriented percolation. More research directions are outlined in the final section of this paper.

\subsection{Overview of the proofs}
The first step in proving our result is to show that $v_p(z)$ does not depend on $z$. This is primarily achieved through game-theoretic arguments, by selecting appropriate sub-optimal strategies and analyzing how the value function varies along the paths generated by these strategies. Next, a classic $0-1$ law argument demonstrates that $v_p$ is deterministic. This last step is more technical than it looks at first glance, as it is not straightforward to show that $v_p(z)$ is measurable with respect to the randomness. 

The proof of our core theorem is more probabilistic. In the supercritical regime $p>p_c$, there exists an infinite oriented path (infinite in both directions) composed solely of open edges (hence, with cost 1). We prove that Player 2 can ensure the token remains within this path, thereby keeping the total cost at 1. 

In the subcritical regime, we prove that Player 1 can guarantee a quantity strictly smaller than one by playing always $T$. Indeed, under such a strategy, the path followed by the token is an oriented path, hence contains a positive density of $0s$, given that $p<p_c$. Surprisingly, to the best of our knowledge, the latter point does not appear in the literature, and we provide a brief proof of it. 

The critical regime $p=p_c$ is more delicate. Indeed, no oriented infinite path of open edges exists in this case. We construct a strategy for Player 2 that navigates between open oriented paths of increasing lengths. Our construction ensures that the time spent by the token moving from one open path to another is negligible compared to the length of the last open path crossed, ensuring the total cost to be 1. This is possible thanks to precise estimates obtained by Duminil-Copin, Tassion and Texeira in \cite{DTT18} regarding the probability of crossings of rectangles in oriented percolation. 

The continuity of the value at $p_c$ is established by constructing a strategy for Player 2 that guarantees a total cost close to 1, provided that $p$ is close to $p_c$. This strategy involves considering boxes of large, fixed size and exploiting the fact that for $p$ near $p_c$, the probability that such a box allows a vertical crossing is close to 1. We demonstrate that Player 1 can navigate between such boxes-those that admit vertical crossings-in a manner that ensures the token spends only a negligible amount of time outside these vertical crossings. The analysis of this strategy once again relies on the estimates in \cite{DTT18}.  
\\
Finally, the fact that $v_p=0$ for $p$ close to 0 is obtained by showing that for such values of $p$, Player 1 can trap the token in an infinite horizontal structure composed entirely of 0s. This structure can be visualized as an infinite horizontal ``thick'' path, where each element corresponds to a square consisting of four edges with a cost of 0.  

The paper is organised as follows. In Section \ref{s.preliminaries}, we present the results from \cite{DTT18}  regarding oriented percolation. In Section \ref{sec:model}, we introduce the game and show that its value is measurable. Section \ref{sec:number} shows that the value is constant. In Section \ref{s.phase_transition}, we establish that the value is 1 if and only if $p$ is larger or equal than $p_c$. In Section \ref{s.continuity}, we show that $v_p$ is continuous at $p_c$. Section \ref{sec:v0} shows that for all $p$ small enough, $v_p=0$. Finally, in Section \ref{sec:persp}, we discuss open problems related to the model.


\section{Preliminaries on oriented percolation}\label{s.preliminaries}
We define the tilted graph
\begin{align*}
	\LL:= \{(x,y) \in \m Z^2: x+y \text{ is even}\},
\end{align*}
where edges exist between points that are exactly a distance of $\sqrt{2}$ apart.

A \textit{path} is a sequence of vertices $(P_i)_{i=\ell_1}^{\ell_2}$, $\ell_1 \in \m{Z} \cup \left\{-\infty\right\}$, $\ell_2 \in \m{Z} \cup \left\{+\infty\right\}$, such that for each $i \in [\ell_1,\ell_2]$, there is an edge between $P_{i}$ and $P_{i+1}$. An \textit{infinite path} corresponds to the case where $\ell_1=-\infty$ and $\ell_2=+\infty$, and a semi-infinite path corresponds to the case where $\ell_1=-\infty$ or $\ell_2=+\infty$. Finally, a \textit{vertical path} is a path $(P_i)_{i=\ell_1}^{\ell_2}$ such that the second coordinate of $P_i$ is monotone.

 We now fix $p\in [0,1]$ and sample a percolation configuration in $\LL$; that is, for each edge  of $\LL$, we assign an i.i.d. Bernoulli random variable with parameter $p$. An edge is said to be closed if its value is $0$, and open if its value is $1$. For $z \in \LL$, we define the event
\begin{align*}
	\vcross_{m,n}(z):= \{z+ (-m,m]\times (-n,n] \text{ is crossed vertically by a vertical open path}\}.
\end{align*}
The border of the symbol $\vcross_{}$ represents a vertical rectangle, and the vertical segment in the middle of the symbol represents the crossing path. Note that for any $z\in \LL$, we have that $\P(\vcross_{m,n}(z))= \P(\vcross_{m,n})$, where $\vcross_{m,n}:=\vcross_{m,n}(0)$. Finally, to simplify notation, when $m$ or $n$ are not integer, we mean $\lfloor m \rfloor$ or $\lfloor n \rfloor$.

As shown in \cite{Dur}, the oriented percolation model undergoes a phase transition in the following sense:
\begin{theorem}\label{t.subcritical} There exists $p_c\in (0,1)$ such that,
\begin{align*}
\begin{cases}
    \P(\vcross_{n,n}) \leq e^{-\gamma_p n} & \text{ if $p<p_c$, for some $\gamma_p >0$}\\
    \P(\vcross_{n, n}) \stackrel{n\to \infty}{\to} 1 & \text{ if } p>p_c.
\end{cases}
\end{align*}
Furthermore, for $p<p_c$, almost surely there are no infinite vertical open paths, and for $p>p_c$, almost surely there exists an infinite vertical open path.
\end{theorem}
\begin{proof}
    The fact that $p_c \in (0,1)$ follows from Sections 3 and 6 of \cite{Dur}.  The exponential decay when $p<p_c$ is from Section 7 of \cite{Dur}. The case when $p>p_c$ comes from the definition of $p_c$ in Section 3 of \cite{Dur}. For the existence of infinite paths, see Section 3 of \cite{Dur}.
\end{proof}
Moreover, a detailed study of the regime $p=p_c$ is presented in \cite{DTT18}.
\begin{theorem}[Theorem 1.2 and 1.3 of \cite{DTT18}] \label{t.connectivity properties}
There exists $c>0$, $\varepsilon>0$ and a sequence $(w_n)_{n\geq 1}$ such that
	\begin{align} \label{e.bound_wn}
		w_n\leq n^{1-\varepsilon} \quad \text{and} \quad \P_{p_c}(\vcross_{w_n,n}) \geq c.
	\end{align}
\end{theorem}
\begin{proof}
	This is a direct consequence of Theorem 1.2 and 1.3 of \cite{DTT18}, along with the comment just below Theorem 1.3 of \cite{DTT18}.
\end{proof}

Theorem \ref{t.connectivity properties} implies the following result. 
\begin{corollary}\label{c.vertical_connections_box}
	For the same constant $c>0$ as in Theorem \ref{t.connectivity properties}, there exist $\sigma<1$ and $\alpha >0$ such that 
	\begin{align*}
		\P_{p_c}(\vcross_{n^\sigma,n})\geq 1-(1-c)^{n^\alpha}\stackrel{n\to \infty}{\longrightarrow} 1.
	\end{align*}
\end{corollary}
\begin{proof}
Consider $\varepsilon>0$ as in \eqref{e.bound_wn}. Take $1-\varepsilon<\eta<1$ and note that if there is no vertical crossing in $(-n^\eta,n^\eta]\times (-n,n]$, there should be no vertical crossing in any box $(x,0) + (-n^{1-\varepsilon},n^{1-\varepsilon}] \times (-n,n]$ contained within $(-n^\eta,n^\eta]\times (-n,n]$. In particular, denoting $E_{m,n}(z)$ the complement of the event $\vcross_{m,n}(z)$, we have
	\begin{align*}
		E_{n^\eta,n}\subseteq \bigcap_{k=-\lfloor n^{\eta-1-\varepsilon}/3 \rfloor}^{\lfloor n^{\eta-1-\varepsilon}/3 \rfloor} \left (E_{n^{1-\varepsilon},n}(3k n^{1-\varepsilon},0) \right ).
	\end{align*}
We conclude by noting that the events $(E_{n^{1-\varepsilon},n}((3k n^{1-\varepsilon},0))_{k}$ are independent, since the corresponding boxes are disjoint.
\end{proof}
	
We can adapt the above corollary to the case where $p$ is slightly smaller than $p_c$.
\begin{corollary}\label{c.p_almost_pc}
For  any $\delta>0$, there exists $n_0\in \m N$ such that the following holds: for all $n\geq n_0$, there exists $p<p_c$,
	\begin{align*}
		\P_{p}(\vcross_{n^\sigma,n})>1-\delta,
	\end{align*}
where $\sigma<1$ is as in Corollary \ref{c.vertical_connections_box}. 
\end{corollary}	
\begin{proof}
Fixing $\delta>0$, one can find $n_0$ such that for all $n\geq n_0$,
\begin{align*}
	\P_{p_c}\left(\vcross_{n^\sigma,n} \right)>1-\delta/2. 
\end{align*}
We conclude by coupling the measures $\P_p$ and $\P_{p_c}$ restricted to the rectangle $(-n^\sigma,n^\sigma]\times (-n,n]$ in such a way that $\P(\omega^p \neq \omega^{p_c})\leq \delta/2$.
\end{proof}

\section{The game model} \label{sec:model}
Let $p \in [0,1]$. We denote by $E$ the set of edges of $\LL$. We consider a collection $(c(e))_{e\in E}$ of independent and identically distributed Bernoulli random variables with parameter $p$. We define a game where Player 1's action set is $\left\{T,B\right\}$, Player 2's action set is $\left\{L,R\right\}$, and that proceeds as follows: 
\begin{itemize}
\item
A token is placed at some initial point $z$ in $\LL$. 
\item
At each stage, Player 1 selects an action and informs Player 2. Then, Player 2 selects an action, and informs Player 1. If $(T,R)$ is played, the token moves to the upper-right vertex. If $(T,L)$ is played, the token moves to the upper-left vertex. If $(B,R)$ is played, then the token goes to the bottom-right vertex, and if $(B,L)$ is played, the token moves to the bottom-left vertex (see Figure \ref{f.diagram_game}). In each case, Player 1 incurs the cost of the corresponding edge.
\end{itemize}
\begin{figure}[h!]
    \centering
    \includegraphics[width=0.2\linewidth]{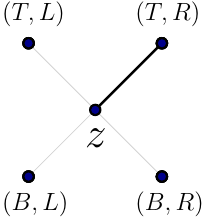}
    \caption{The figure represents a turn of the game. The token is at position $z$; gray edges and black edges represent edges with costs equal to $0$ and $1$, respectively. An edge is crossed if the corresponding pair of actions is chosen, after which the token moves to the other vertex on the edge.}
    \label{f.diagram_game}
\end{figure}

We consider the infinite game where Player 1 aims to minimize costs, and Player 2 aims to maximize them. The \textit{history} at stage $m$ is the sequence of edges $(e_1,\dots,e_{m-1})$ that the token has crossed before stage $m$: this represents Player 1's information at the start of stage $m$. The set of possible histories is $H:=\cup_{m \geq 1} E^{m-1}$. A  strategy for Player 1 is a measurable mapping $\sigma$ that associates an element $\sigma(h) \in \left\{T,B\right\}$ to each possible history $h$. A strategy for Player 2 is a mapping that associates an element of $\tau(h) \in \left\{L,R\right\}^2$ to each possible history $h$, with the following interpretation: the first component of $\tau(h)$ is the action chosen when Player 1 plays $T$, and the second component of $\tau(h)$ is the action chosen when Player 1 plays $B$. The set of strategies for Player 1 is denoted by $\Sigma$, and the set of strategies for Player 2 is denoted by $\mathcal{T}$. 
The total cost induced by a pair of strategies $(\sigma,\tau)$ is defined as
\begin{equation*}
\gamma(z,\sigma,\tau)=\limsup_{n \rightarrow + \infty} \frac{1}{n} \sum_{m=1}^n c(e_m).
\end{equation*}
The game starting from $z$ is denoted by $\Gamma(z)$. 
 A consequence of the work in \cite{MS} is the following proposition.
\begin{proposition}\label{p.measurability}
For each $z \in \LL$, the game $\Gamma(z)$ has a \emph{value}, denoted by $v(z)$:
     \begin{align*}
    \inf_{\sigma \in \Sigma} \sup_{\tau \in \mathcal{T}} \gamma(z,\sigma,\tau)=\sup_{\tau \in \mathcal{T}} \inf_{\sigma \in \Sigma}\gamma(z,\sigma,\tau):=v(z).
     \end{align*}
   Moreover, the function $(c(e))_{e\in E} \mapsto v(z)$ is measurable. 
 \end{proposition}
 \begin{proof}
     Theorem (1.1) of \cite{MS} concerns the measurability of the value function in stochastic games with \emph{$\limsup$ cost}, and our game can be embedded in such a class. Specifically, using the notation from \cite{MS}, we define the following:
     \begin{itemize}
         \item The state space is $X=[0,1]^{E}\times \LL \times \mathbb N^*\times [0,1] $, the action set for Player 1 is $A=\{-1,1\}$, and the action set for Player 2 is $B=\{-1,1\}^2$. The functions $F$ and $G$ are constant and equal to $A$ and $B$, respectively.
         \item The transition function $q:X\times A\times B \mapsto X$ is a (deterministic) function that acts as follows:
         \begin{align*}
         q((\omega, z , n , v),a,b) = \left (\omega, z+e, n+1,  \frac{n }{n+1 } v + \frac{1}{n+1} \omega_{z,z+e} \right ),
         \end{align*}
         where $e$ is $(-1,b_1)$ or $(1,b_2)$ if $a$ is $-1$ or $1$, respectively.
         \item The cost function $u: X \mapsto [0,1]$ is the projection to the last coordinate.
     \end{itemize}

All conditions stated in (1.2) of \cite{MS} are satisfied, and thus Theorem 1.1 of \cite{MS} implies that the value function is measurable.
 \end{proof}

 \begin{remark}
  In this paper, measurability is always with respect to the completed $\sigma$-algebra generated by the random variables $(c(e))_{e\in E}.
     $. However, Theorem 1.1 of \cite{MS} implies that the value function is, in  fact, upper analytic and thus universally measurable. As a result,  Proposition \ref{p.measurability} remains valid for the completed $\sigma$-algebra of any probability measure on $(c(e))_{e\in E}\in [0,1]^E$.
 \end{remark}
We can now define the standard notions of \textit{guarantee} and \textit{$\varepsilon$-optimal strategy}. 
\begin{definition}
Player 1 guarantees $w \in \m{R}$ in $\Gamma(z)$ if there exists $\sigma \in \Sigma$ such that for all $\tau \in \mathcal{T}$, $\gamma(z,\sigma,\tau) \leq w$. We will also say that $\sigma$ \textit{guarantees} $w$. 
\\
Player 2 guarantees $w \in \m{R}$ in $\Gamma(z)$ if there exists $\tau \in \mathcal{T}$ such that for all $\sigma \in \Sigma$, $\gamma(z,\sigma,\tau) \geq w$. We will also say that $\tau$ \textit{guarantees} $w$. 
\end{definition}
\begin{definition}
Let $\varepsilon>0$. A strategy $\sigma \in \Sigma$ is \emph{$\varepsilon$-optimal} in $\Gamma(z)$ if it satisfies that for all $\tau \in \mathcal{T}$, $\gamma(z,\sigma,\tau) \leq v(z)+\varepsilon$. A strategy $\tau \in \mathcal{T}$ is \emph{$\varepsilon$-optimal} in $\Gamma(z)$ if it satisfies that for all $\sigma \in \Sigma$, $\gamma(z,\sigma,\tau) \geq v(z)-\varepsilon$.
\end{definition}
\section{The value is a number} \label{sec:number}
Our first result states that the value does not depend neither on the initial position of the token, nor on the realizations of the costs.

\begin{theorem} \label{theo:det}
The random variable $v(z)$ does not depend on $z$ and is deterministic.
\end{theorem}
We start with the following technical lemma, which holds for any possible realization of the cost $(c(e))_{e\in E}$.
\begin{lemma} \label{lemma:shift}
Let $z \in \LL$, $(\sigma,\tau)$ be a pair of strategies, $h$ be some history at stage $M \geq 1$, and $z'$ be the vertex reached after history $h$. Then, 
$\gamma(z,\sigma,\tau)=\gamma(z',\sigma[h],\tau[h])$, 
where $\sigma[h] \in \Sigma$ and $\tau[h] \in \mathcal{T}$ are defined for all history $h'$ by $\sigma[h](h'):=\sigma(hh')$ and $\tau[h](h'):=\tau(hh')$. 
\end{lemma}
\begin{proof}
Let $(e_m)_{m \geq 1}$ be the sequence of edges induced by $\sigma$ and $\tau$, starting from $z$, and $(e'_m)_{m \geq 1}$ be the sequence of edges induced by $\sigma[h]$ and $\tau[h]$, starting from $z'$. By definition, we have $e'_m=e_{m+M-1}$. Hence,
\begin{eqnarray*}
\gamma(z',\sigma[h],\tau[h])&=&\limsup_{n \rightarrow + \infty} \frac{1}{n} \sum_{m=1}^n c(e'_m)
\\
&=& \limsup_{n \rightarrow + \infty} \frac{1}{n} \sum_{m=1}^n c(e_{m+M-1})
\\
&=& 
\limsup_{n \rightarrow + \infty} \frac{1}{n} \sum_{m=1}^n c(e_{m})
\\
&=& \gamma(z,\sigma,\tau).
\end{eqnarray*}
\end{proof}
With this lemma, we are now ready to prove the main theorem of this section.
\begin{proof}[Proof of Theorem \ref{theo:det}] We begin by showing that $v(z)$ does not depend on $z$. To do this, take $\varepsilon>0$ and $z, z' \in \LL$, without loss of generality, assume $z'$ is to the right of $z$. Let $\widehat{T}$ be the strategy of Player 1 that always play $T$, irrespective of the history. Let $\tau$ be an $\varepsilon$-optimal strategy for Player 2 in the game starting from $z$. Let $P_T$ be the set of vertices that are visited by the token when $(\widehat{T},\tau)$ is played (see Figure \ref{f.independence_z}).

We, first claim that for all $\hat z\in P_T$, $v(\hat z) \geq v(z)-2\varepsilon$. Indeed, let $\hat z\in P_T$, $M$ be the first instant where $\hat z$ is reached under strategies $(\widehat{T},\tau)$, and $h$ be the history at stage $M$. Consider the following strategy $\sigma$ of Player 1: play $T$ until reaching $\hat z$, then play an $\varepsilon$-optimal strategy in the game $\Gamma(\hat z)$\footnote{Recall that $\Gamma(\hat z)$  is the game starting from $\hat z$. }. By Lemma \ref{lemma:shift}, we have 
$\gamma(z,\sigma,\tau)=\gamma(\hat z,\sigma[h],\tau[h])$. Because $\sigma[h]$ is $\varepsilon$-optimal in $\Gamma(\hat z)$, we have $\gamma(\hat z,\sigma[h],\tau[h]) \leq v(\hat z)+\varepsilon$, hence 
$\gamma(z,\sigma,\tau) \leq v(\hat z)+\varepsilon$. Since $\tau$ is $\varepsilon$-optimal in $\Gamma(z)$, we have $\gamma(z,\sigma,\tau) \geq v(z)-\varepsilon$, and we deduce that $v(\hat z) \geq v(z)-2\varepsilon$ for all $\hat z \in P_T$. 

 Analogously, we now construct two other paths: $P_B$, the set of vertices visited by the token when the game starts in $z$, Player 1 always play $B$, and Player 2 plays some $\varepsilon$-optimal strategy; and $P'$, the set of vertices visited by the token when the game starts in $z'$ instead of $z$, Player 2 always plays $L$, and Player 1 plays an $\varepsilon$-optimal strategy (see again Figure \ref{f.independence_z}). The same argument as above shows that for any $\hat z \in P_B$, $v(\hat z) \geq v(z)-2 \varepsilon$ and  $v(\tilde z) \leq v(z')+2 \varepsilon$ for all $\tilde z \in P'$.

 \begin{figure}[h!]
    \centering
    \includegraphics[width=0.25\linewidth]{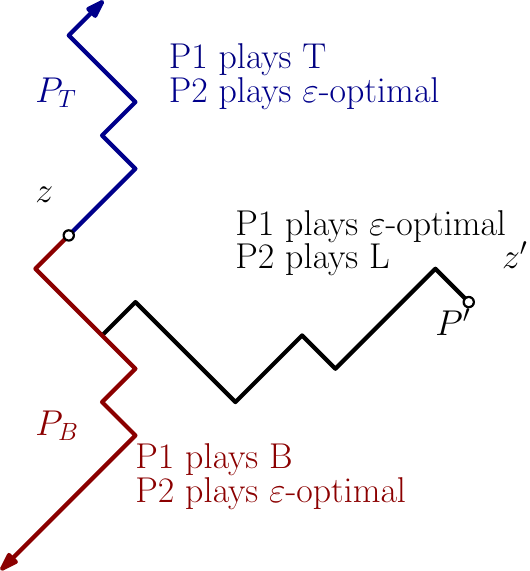}
    \caption{Diagram of the three paths used in the proof. In this case $P'$ intersect $P=P_B\cup P_T$.}
    \label{f.independence_z}
\end{figure}

We now use the inequalities obtained to conclude. First, set $P=P_B \cup P_T$ and consider the case where $\emptyset \neq P\cap P'\ni z_0$. Then, $v(z)-2\varepsilon \leq v(z_0) \leq v(z')+2\varepsilon$, which implies that $v(z)-4\varepsilon<v(z')$. 

We are left with the case where $P \cap P'=\emptyset$. Then, $P'$ contains either a finite number of vertices at which Player 1 plays $B$, or a finite number of vertices at which Player 1 plays $T$. 
 Without loss of generality, assume that we are in the first case. Note that in this case, $P_T$ has to contain only a finite number of $R$. Hence, by the law of large numbers the total cost along $P_T$ and along $P'$ is almost surely $p$. We deduce that 
 $v(z)-\varepsilon \leq p$ and $v(z') + \varepsilon \geq p$, hence $v(z) -2\varepsilon \leq v(z')$. 
 
 As $\varepsilon>0$ is arbitrary, the above arguments show that if $z'$  is on the right side of $z$, then $v(z) \leq v(z')$. By symmetry of the game with respect to the vertical axis, we deduce that for all $z$ and $z'$, $v(z)=v(z')$. 

Once we know that $v$ is independent of $z$, for any $c\in \mathbb R$, the event $\left\{v \leq c\right\}$ must have probability $0$ or $1$ as it is invariant by translation\footnote{This is a consequence of the ergodicity of translations of i.i.d environments, see for example page 38 of \cite{KL} or Lemma 2.8 of \cite{DCLN}.}. Thus, $v$ is almost surely deterministic. 
\end{proof}
\section{Phase transition of the value at $p_c$}\label{s.phase_transition}
In this section, we relate the value of our game with $p_c$, the critical parameter of oriented percolation. This relationship is established in the following theorem.
\begin{theorem} \label{theo:value_threshold}
$v_p=1$ if and only if $p \geq p_c$. 
\end{theorem}
\begin{proof}
We will consider three cases. The first one is included in the second but is presented separately for clarity.
\subsection*{Case 1: $p>p_c$}
Consider an infinite vertical path $P$, which exists almost surely by Theorem \ref{t.subcritical}. Let us prove that if $z$ is in $P$, then $v(z)=1$. By Theorem \ref{theo:det}, $v$ does not depend on the initial position almost surely, hence this is enough to prove that $v=1$. Consider the following strategy for Player 2: if the current position $z$ is in $P$, and Player 1 plays some action, Player 2 should play in such a way that the next position of the token remains in $P$. This strategy ensures that the token always remains in $P$, hence it guarantees a total cost 1: $v=1$ (see Figure \ref{f.infinite_strategy}).
\begin{figure}[h!]
    \centering
    \includegraphics[width=0.3\linewidth]{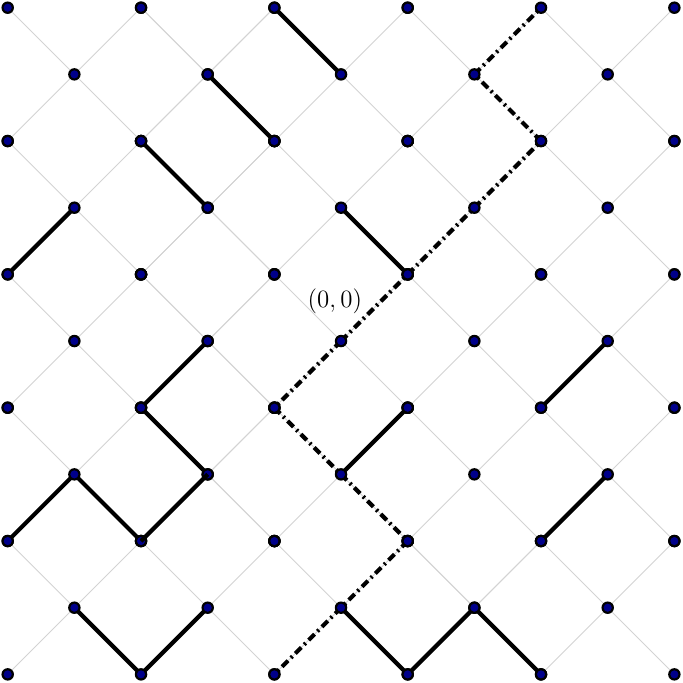}
    \caption{The figure  shows the graph $\LL$ with its costs. The black edges have a cost of $1$, the gray ones have a cost of $0$, and the dashed line is a crossing path from top to bottom. The game starts at $(0,0)$, which already belongs to a vertical crossing. Note that Player 2 always has a strategy to remain in the crossing. For example, in the first move, if Player 1 plays $T$, Player 2 chooses $R$, and if Player 1 plays $B$, Player 2 chooses $L$.}
    \label{f.infinite_strategy}
\end{figure}
\subsection*{Case 2: $p=p_c$}
Let $n \geq 1$. By Borel-Cantelli and Corollary \ref{c.vertical_connections_box}, there exists $\sigma<1$, almost surely, there exists $n_0 \geq 1$ such that for all $n \geq n_0$, the event $\vcross_{2^{n\sigma},2^n}$ happens. Define $P^n$ as the left-most vertical path that crosses $(-2^{n\sigma},2^{n\sigma}]\times (-2^{n},2^{n}]$, noting that $P^n$ is at distance at most $2^{-n\sigma}$ from the origin. Since $v$ is constant, we can assume w.l.o.g. that the initial position lies in the middle of $P^{n_0}$. Define $m_0:=1$ and for each $k \geq 1$:
$$
m_k:=1+\sum_{\ell=n_0}^{n_0+k-1} 2^{\ell-1}.
$$
We now build recursively a strategy for Player 2 (see Figure \ref{f.Crossings}) that satisfies the following properties, regardless of Player 1'strategy and for each $k \geq 0$:
\begin{enumerate}
\item
Between stages $m_k$ and $m_{k+1}-1$, the token spends at most $2^{(n_0+k) \sigma+1}$ steps on edges with cost 0. 
\item 
At stage $m_{k+1}$, the position lies in $P^{n_0+k}$. 
\end{enumerate}
The first property readily implies that such a strategy guarantees a total cost of 1, while the second property is useful for the induction step. 

\paragraph{Initial step: $k=0$.}
Starting from $n_0$, Player 2 follows the same strategy as in Figure \ref{f.infinite_strategy} to remain in $P^{n_0}$ until stage $m_1-1=2^{n_0-1}$. By the definition of $P^{n_0}$, both properties 1. and 2. are satisfied. 
\paragraph{Induction step $k \geq 1$.}
If the token is at $P^{n_0+k}$ at stage $m_{k}$, Player $2$ follows the same strategy as in Figure \ref{f.infinite_strategy}, ensuring the token remains in $P^{n_0+k}$ until stage $m_{k+1}$. If not, Player 2 chooses the direction that brings the token closer to $P^{n_0+k}$, reaching it in at most $2^{(n_0+k-1)\sigma}+2^{(n_0+k)\sigma} \leq 2^{(n_0+k)\sigma+1}$ steps, that is, the time to reach the vertical axis and then to go to $P^{n_0+k}$ from the vertical axis.  Once the token reaches $P^{n_0+k}$, Player $2$ follows the strategy in Figure \ref{f.infinite_strategy} to keep the token in $P^{n_0+k}$ until stage $m_{k+1}$. Thus, Property 2 is satisfied, and Property 1 holds since the token reaches $P^{n_0+k}$ before stage $m_k+2^{(n_0+k) \sigma+1}$ and then stays in $P^{n_0+k}$, where the cost is 1. 
\\

Therefore, we have constructed a strategy for Player 2 such that, for any strategy of Player 1, in each period $\left\{m_k,\dots,m_{k+1}-1\right\}$, the token spends at most $2(m_{k+1}-m_k)^{\sigma}$ stages on edges with cost 0. Since $\sigma<1$, such a strategy guarantees a total cost 1 to Player 2, so $v=1$. 
\begin{figure}[h!]
    \centering
    \includegraphics[width=0.2\linewidth]{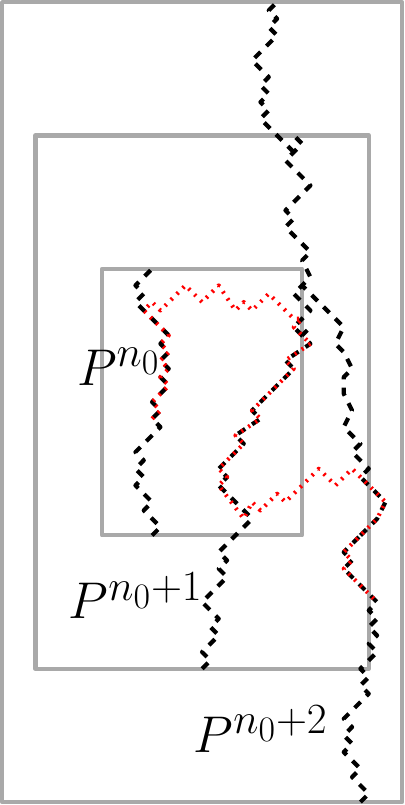}
    \caption{The gray boxes represent the boxes $(-2^{n\sigma},2^{n\sigma}]\times (-2^n,2^n]$. The black paths are crossings, and the red path represents the path followed by the token under the strategy constructed for Player 2. Note that the red path depends on Player 1's strategy.}
    \label{f.Crossings}
\end{figure}

\subsection*{Case 3: $p<p_c$}
Consider the strategy of Player 1 that plays $T$ at every stage. Regardless of Player 2's strategy, the path followed by the token is a semi-infinite vertical path. To conclude, we only need to verify that any such path $P$ has a total cost smaller than $1$. This is done in the following claim.
\begin{claim}\label{c.last_passage}
There exists $\delta >0$ such that
    \begin{align*}
    \sup_{P} \limsup_{n\to \infty}\frac{1 }{n }\sum_{m=1}^n c(P_m) <1-\delta,
    \end{align*}
    where the supremum is taken over all semi-infinite vertical paths starting at $0$, and $P_m$ designates the $m$-th vertex of $P$. 
\end{claim}
\begin{proof}
   Take $n$ large so that $\P_p(\vcross_{n,n})<(300n)^{-1}$ (which is possible thanks to Theorem \ref{t.subcritical}). Now, tesselate $\mathbb Z^2$ with squares of side length $n$, with the first one centered at $0$. For any semi-infinite path $P$ starting at $0$ and $k \geq 0$, denote by $B_k$ the first box with height $2k n$ that is reached by $P$. We refer to such a collection $(B_k)$ as a box-path. The path enters each box $B_k$ at time $m_k=n + 2(k-1)n$ (see Figure \ref{f.Box_projection}).  

Let $A_k$ be the event $$A_k:=\{\text{ There is no open vertical path from the bottom of $B_k$ to the bottom of $ B_{k+1}$} \}.$$ 
   If $A_k$ occurs, then the intersection of $P$ with box $B_k$ must contain at least one edge with cost 0. This implies
   \begin{align*}
\sum_{m=m_k}^{m_k +n-1}(1-c(P_m))\geq  \1_{A_k}.
   \end{align*}
   Moreover, for any $k$,
   \begin{align}\label{e.ineq_cost}
   \nonumber&\P\left (\sup_P \frac{1 }{m_k }\sum_{m=1}^{m_k} c(P_m) > 1- \frac{1}{4n }\right ) \\
   \nonumber&\leq \sum_{\substack{(B_i)_{i=1}^k\\ B \text{ is a box path}}} \P(\text{At least $k/2$, $j\in \{1,.. k \}$ are s.t. $A_j^c$ occurs for $B$})\\
   &\leq 3^{k} \times 2^{k} \times (3n\P_p(\vcross_{n,n}))^{k/2}< 2^{-k}.
   \end{align}
   Note that the second inequality in the above equation was derived as follows: the term $3^k$ arises because any vertical path has $3$ choices for the next square, the factor $2^{k}$ bounds the number of pairs that can be taken, and the probability that $A_j^c$ occurs is upper-bounded using the union bound by $3n\P_p(\vcross_{n,n})$. We conclude the claim from \eqref{e.ineq_cost} using the Borel-Cantelli lemma. 
   \begin{figure}
       \centering
       \includegraphics[width=0.27\linewidth]{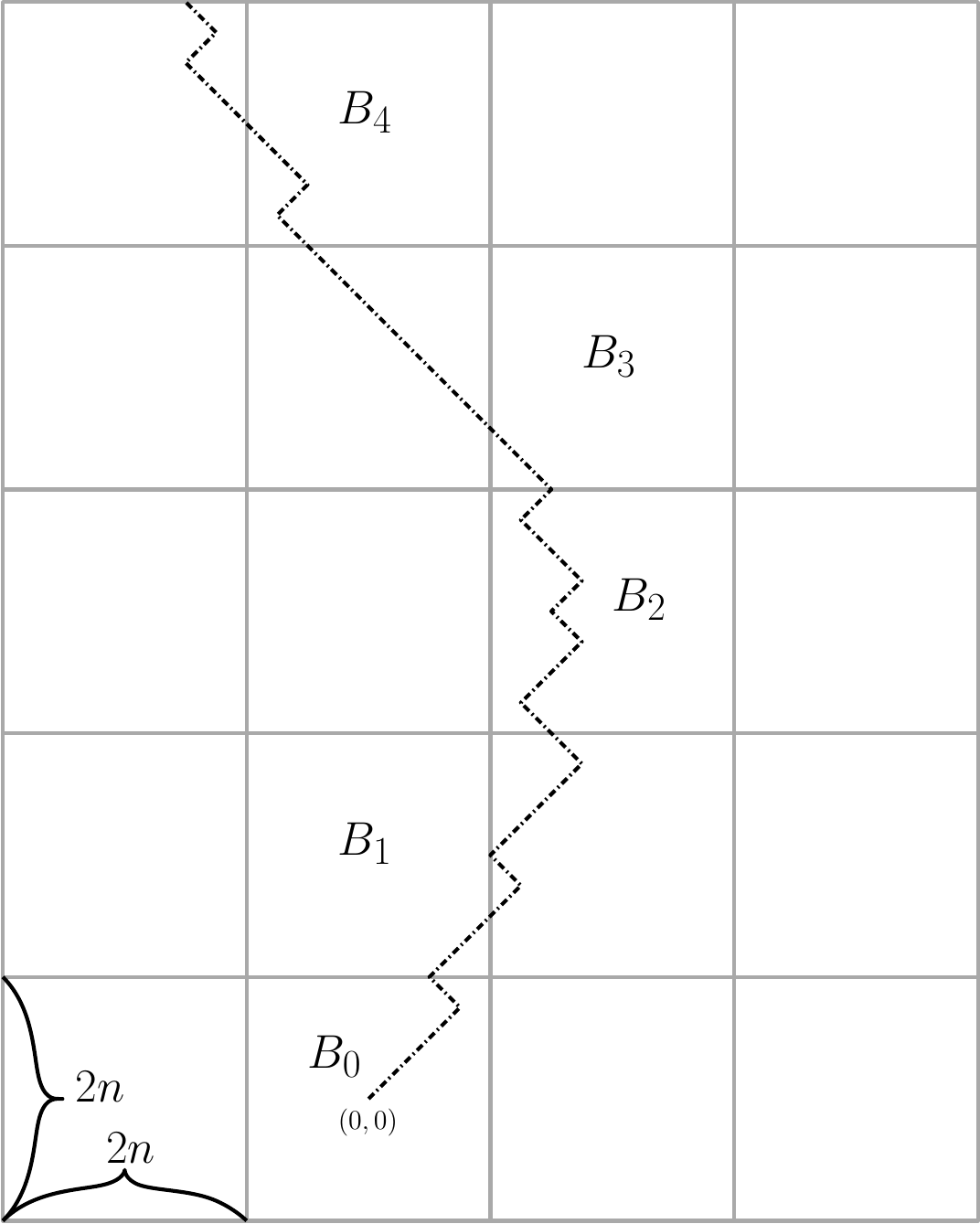}
       \caption{A representation of a semi-infinite path with its projection in the boxes. Note that the path starts in the middle of $B_0$.}
       \label{f.Box_projection}
   \end{figure}
\end{proof}
\end{proof}

To conclude this section, let us remark that Claim \ref{c.last_passage} is analogous to a shape result obtained in first passage percolation (see, for example, Section 2.3 of \cite{FPP}). However, we could not find a reference for this result in the context of last passage percolation. As the proof here is simpler than in the first passage percolation case, we chose to include it.

\section{Continuity of the value at $p_c$}\label{s.continuity}
In this subsection, we study the continuity of the value function at $p_c$.
\begin{theorem} \label{theo:continuity}
The function $p\mapsto v_p$ is continuous at $p_c$. 
\end{theorem}

Before proving the theorem, note that since $v_p=1$ for any $p\geq p_c$, it suffices to prove that $\liminf_{p\nearrow p_c} v_p=1$.

\begin{proof}
    The proof relies crucially on Corollary \ref{c.p_almost_pc}.  We begin by choosing $0<\delta <1/25$ and $\sigma<1$ such that for all $n\geq n_0(\delta)$, there exists $p_0<p_c$ such that for any $p>p_0$:
    \begin{align*}
    \P_{p}(\vcross_{m,n})>1-\delta,
    \end{align*}
    where $m=\lfloor n^\sigma \rfloor$. Throughout this proof, we work exclusively with boxes of the form $x + (-m,m]\times(-n,n]$ with $x\in (m\mathbb Z)\times(n\mathbb Z)$. We say a point is ``good'' if the event $\vcross_{m,n}(x)$ occurs \footnote{Note that the random variables $(\vcross_{m,n}(x))_{x\in \in (m\mathbb Z)\times(n\mathbb Z)}$ are not independent as each rectangle intersects two others, thereby sharing some of the same random variables. However, they are independent as long as they are not vertical neighbours, i.e., as long as they do not intersect.} (see Figure \ref{f.NewGrid}). 

    \begin{figure}[h!]\centering
\includegraphics[width=0.14\textwidth]{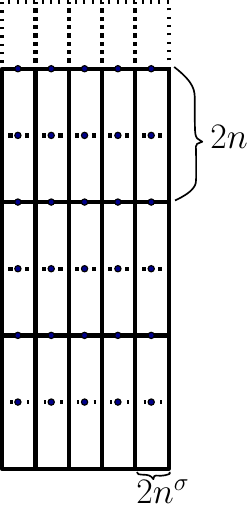}
\caption{This figure shows the new grid that emerges when we examine vertical boxes. The dots represent the center of the boxes in $(m\mathbb Z)\times (n\mathbb Z)$, while the solid lines or dotted lines represent the possible boxes. Note that any given point belongs to only two boxes, and boxes that are not on the same vertical slit do not intersect.}
\label{f.NewGrid}
\end{figure}
    
    We now establish a recursive strategy for Player 2 that guarantees a total cost close to 1. We start at $x=0$ and depending on whether $x$ is good or not, Player 2 will follow a different strategy (See Figure \ref{f.PStrategie} for a possible result of this strategy):
    \begin{itemize}
        \item If $x$ is good, Player 2 will choose a constant direction that moves the token closer to the nearest vertical crossing of $x + (-m,m]\times(-n,n]$. When the token reaches the crossing, Player 2 will take the necessary actions to remain on the crossing until the token exits the box, either through the top or the bottom (this can be achieved using the strategy explained in Figure \ref{f.infinite_strategy}). Note that in this case, the token spends at least $n/2-2n^\sigma$ steps on edges with a cost of $1$, and at most $2n^\sigma$ steps on edges with a cost of $0$.
        \item If $x$ is bad, Player 2 will move the token in the right direction. In this case, the token will spend at most $2n^\sigma$ steps on edges with a cost of $0$.
    \end{itemize}
    After any of these steps, the token exits the box $x + (-m,m]\times(-n,n]$. This means the token will now be in two different boxes. At least one of these boxes will have a distance from the vertical limits greater than or equal to $n/2$, we call $x$ the center of this box and repeat the procedure described above.
    \begin{figure}[h!]\centering
\includegraphics[width=0.25\textwidth]{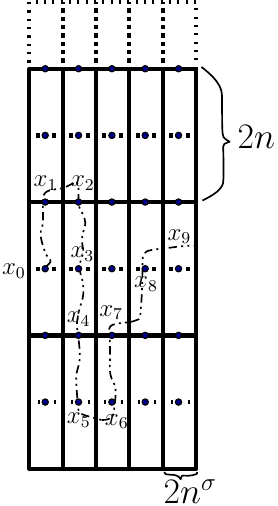}
\caption{This figure illustrates a possible outcome of the strategy described. The point $x_0=(0,0)$ is where the game started; since it was a good box, the token exited through the top. After the exit, the token belongs to two different boxes: the one centered at $x_1=(0, n)$ and the one centered at $x_2=(0,2n)$. Player 2 chooses $x=x_1$ and realizes it is a bad box, so the token exits to the right. The process continues until the token exits through the bad box centered at $x_9$. The points $x_0$, $x_2$, $x_3$, $x_4$, $x_6$ and $x_8$ are good, while $x_1$, $x_5$ and $x_9$ are bad. }
\label{f.PStrategie}
\end{figure}

    In this strategy, the token never revisits bad boxes represented by the same $x\in  (m\mathbb Z)\times(n\mathbb Z)$. Assume that it has passed through $N$ different boxes $(x_i)_{i=1}^N$; the cost is then lower bounded by
    \begin{align*}
    1-\frac{N 2n^{\sigma} }{\sum_{i=1}^{N} \mathrm{time}(x_N) } \geq 1-\frac{2Nn^{\sigma} }{(\frac{n }{2 }-2n^\sigma)\sum_{i=1}^N \1_{x_i \text{ is good}} },
    \end{align*}
    where time($x_N)$ denotes the number of turns spent between exiting $x_{N-1}$ and exiting $x_{N}$. We now define the random variable $\Tilde N = \sum_{i=1}^N \1_{x_i \text{ is good}}$ and study
    \begin{align*}
    \P( \Tilde N \leq N/2) \geq \delta^{N/2}\sharp\{(x_{i})_{i=1}^N: \text{ acceptable path with at least $N/2$ right turns} \} \leq \delta^{N/2} 5^{N-1}.
    \end{align*}
    Here, an acceptable path is one in $(-m,m]\times[-n,n]$ where each step moves either up, down, right, up-right, or down-right. This path represents a possible path that the players took in the game using the strategy described for Player 2. Note that each time there is a right, up-right or down-right step, it must be from a bad box, and the token never returns to that box again.
    
    Since $\delta < 1/25$, and by using Borel-Cantelli, we see that eventually $\tilde N > N/2$, leading to the following bound: 
    \begin{align}\label{e.bound_limit}
    \limsup_{n' \rightarrow +\infty} \frac{\sum_{k=1}^{n'}c(e_k) }{n'}\geq 1- 4n^{\sigma-1} \liminf_{N \rightarrow+\infty} \frac{N }{\tilde N } \geq 1-2n^{\sigma-1}.
    \end{align}

    We conclude by noting that \eqref{e.bound_limit} implies $\liminf_{p\nearrow p_c} v_p(0)\geq 1-n^{\sigma-1}$ for any $n\in \mathbb N$.
    
\end{proof}

\section{Necessary condition for $v=0$} \label{sec:v0}
Consider a square composed of four neighbouring edges. The square is called a \textit{0-square} if all four edges are 0. A \textit{0-path} is an infinite horizontal path composed entirely of 0-squares.

Let $p_0\in [0,1]$ be the critical point corresponding to the following event: there exists a 0-path. 
The following theorem states that below $p_0$, we have $v_p=0$.
\begin{theorem} \label{t.v=0}
For all $p<p_0$, $v_p=0$. 
\end{theorem}
\begin{proof}
Player 2 can reach the 0-path in a finite number of stages and then keep the token in it, thereby guaranteeing a total cost of 0 (see Figure \ref{f.0path_strategy}).
\begin{figure}[h!]
    \centering
    \includegraphics[width=0.3\linewidth]{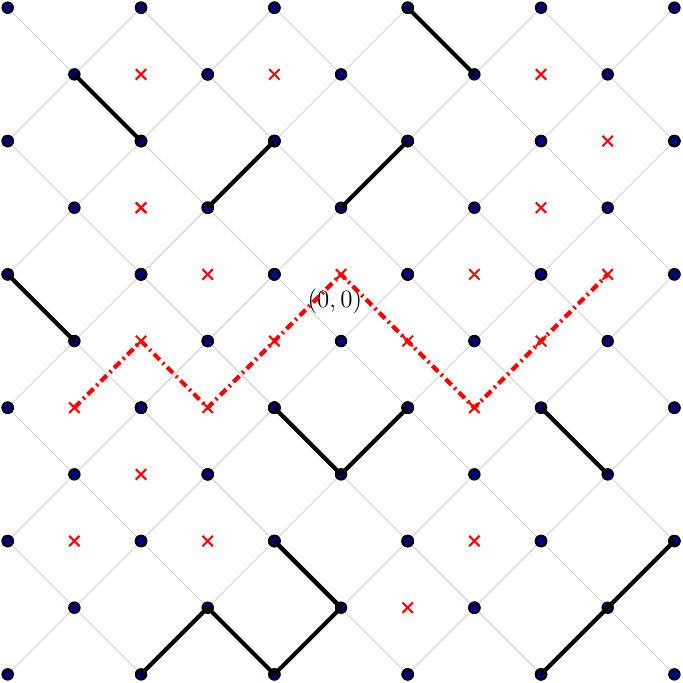}
    \caption{The figure illustrates the graph $\LL$ with its associated costs. Edges where the cost is 1 are shown in black, while those where the cost is 0 are shown in gray. Crosses mark the centers of $0$-squares. The game starts at $(0,0)$, which already belongs to an infinite horizontal crossing of $0$-squares. Note that Player 2 always has a strategy to remain on the crossing by choosing the direction that points toward the cross. }
    \label{f.0path_strategy}
\end{figure}
\end{proof}
We conclude this section by showing, using Peierls' argument, that $p_0$ is indeed non-trivial.
\begin{proposition}
    We have that $p_0\in (0,1)$, and for any $p<p_0$, $\P_p$-almost surely, there exists a $0$-path.
\end{proposition}
\begin{proof}
    Clearly, $p_0<1$ because the value is monotone with respect to $p$, and thus $p_0<p_c<1$. To see that $p_0>0$, we need a Peierls-type argument. We identify squares with their centers ($(x,y)$ such that $x+y$ is odd), and note that the problem of the existence of paths is now a problem of oriented site percolation with $2$-dependence (i.e., random variables that are at a distance greater than $2$ are independent). 
    
    If there is no $0$-path starting from $(0,1)$, then there must be some $n\in \mathbb N^*$ such that there is a dual top-to-bottom crossing of the square $(0,1)+[-n,n]\times[n,n]$ by a non-self-intersecting path\footnote{In this context, a dual crossing is a path on faces where at each step, the intersection between the closures of two faces is not trivial, i.e., meaning this path may include ``diagonal" edges.}. Using the union bound, we see that this probability is upper bounded by
    \begin{align*}
    \sum_{n\in \mathbb N} \sum_{ \substack{P:\\ \text {starts at $[n,n]\times \{1-n\}$}\\ \text{ends at $[n,n]\times \{n+1\}$ }}}\P(\text{No square of $P$ is a $0$-square}) \leq \sum_{n\in \mathbb N} \sum_{l\geq 2n} n^2 5^l (1-(1-p)^4)^{l/5} \stackrel{p\to 0}{\to} 0.
    \end{align*}
   Here, the term $n^2 5^l$ bounds the amount of non-self-intersecting paths of length $l$, $(1-(1-p)^4)$ is the probability that a given box is not a $0$-square and that each self-avoiding path of length $l$ has at least $l/5$ boxes who do not share any edges. We conclude by taking $p>0$ small enough such that the above upper bound is smaller than $1/2$, in which case we have a positive probability of having a $0$-path. Since the event of having a $0$-path is translation-invariant, we conclude that the probability must be either $0$ or $1$.
\end{proof}

\begin{remark}
 Note that a $0$-path is not the only geometric construction that allows player 1 to have a strategy yielding only $0$s. A simpler (but unlikely) geometric structure is an infinite horizontal \textit{zig-zag path}, which is an infinite horizontal path of edges that alternates between going up and down.    
\end{remark}

\section{Perspectives} \label{sec:persp}
We present here a few open questions that we believe could be of interest to both probabilists and game theorists. 

\begin{itemize}
    \item 
    We proved that the function $p \rightarrow v(p)$ is continuous at $p_c$. The numerical simulation presented in Figure \ref{f.graph} suggest that $v$ is continuous everywhere. Is this true? 
    \item 
    By analogy with first-passage percolation and last-passage percolation, one can define the $n$-stage game where the total cost is the average $\frac{1}{n} \sum_{m=1}^n c(e_m)$ and consider its value $v_n$. This type of game is also well-studied in the game theory community. Is it true that $v_n$ converges almost surely, and does it converge to the value $v$ defined in this paper?. The main difficulty in this questions lies in the fact that one cannot a priori use subadditive arguments to show the existence of such a limit as in the case of first and last passage percolation.
    \item 
    A related question is whether defining the total cost as $\liminf_{n \rightarrow +\infty} \frac{1}{n} \sum_{m=1}^n c(e_m)$ instead of $\limsup_{n \rightarrow +\infty} \frac{1}{n} \sum_{m=1}^n c(e_m)$ changes the value of the game. 
    \item 
    We proved in Section \ref{sec:v0} that the existence of a 0-path is enough to guarantee that $v=0$. We also noted that this was not a necessary condition. Can we characterize the type of path such that $v=0$ if and only if such a path exists? 
\end{itemize}

\section*{Acknowledgments}
The authors are grateful to Luc Attia, Lyuben Lichev, Dieter Mitsche and Raimundo Saona, for valuable discussions. The game we consider in this paper is inspired by an example that we discussed with Raimundo Saona and Luc Attia, and which is detailed in Chapter 7 of Luc Attia's PhD thesis. 
We are grateful to Melissa Gonzalez Garcia for encoding our game in Python and making the numerical simulations that produced Figure \ref{f.graph}. The research of A.S.\,is supported by grants ANID AFB170001, FONDECYT regular 1240884 and ERC 101043450 Vortex, and was supported by FONDECYT iniciaci\'on de investigaci\'on N$^o$ 11200085. The research of B. Z. was supported by the French Agence Nationale de la Recherche (ANR) under reference ANR-21-CE40-0020 (CONVERGENCE project). Part of this work was completed during a 1-year visit of B.Z. to the Center for Mathematical Modeling (CMM) at University of Chile in 2023, under the IRL program of CNRS.  
\bibliographystyle{alpha}
\bibliography{bibliogen}
\end{document}